\documentclass[12pt]{amsart}
\usepackage{geometry}                
\usepackage{graphicx}
\usepackage{amssymb}
\usepackage{epstopdf}
\usepackage{amsmath}
\usepackage{color}
\usepackage{adjustbox}
\usepackage{mathtools}
\usepackage{enumitem}
\usepackage{tikz-cd}
\usepackage{hyperref}
\usepackage[dvipsnames]{xcolor}

\usepackage[all]{xy}
\xyoption{matrix}
\xyoption{arrow}

\definecolor{darkred}{RGB}{180, 0, 0}

\usetikzlibrary{arrows,decorations.pathmorphing,backgrounds,positioning,fit,petri}

 \tikzset{help lines/.style={step=#1cm,very thin, color=gray},
help lines/.default=.5} 
\tikzset{thick grid/.style={step=#1cm,thick, color=gray},
thick grid/.default=1} 

\newtheorem{thm}{Theorem}[section]
\newtheorem{lem}[thm]{Lemma}

\newtheorem{prop}[thm]{Proposition}

\theoremstyle{definition}
\newtheorem{defn}[thm]{Definition}

\theoremstyle{remark}
\newtheorem{rem}[thm]{Remark}

\numberwithin{equation}{section}

\newcommand{\noans}[1]{}

\DeclareMathOperator{\Hom}{Hom}%
\DeclareMathOperator{\Ext}{Ext}%
\newcommand{\kk}{\ensuremath{\Bbbk}}

\newcommand{\commentout}[1]{}

\newcommand{\cD}{\ensuremath{{\mathcal{D}}}}

\newcommand{\cH}{\ensuremath{{\mathcal{H}}}}

\newcommand{\cK}{\ensuremath{{\mathcal{K}}}}

\newcommand{\cS}{\ensuremath{{\mathcal{S}}}}

\definecolor{jg}{RGB}{150,50,0}

\definecolor{rm}{RGB}{0,0,200}

\definecolor{ab}{RGB}{100,150,200}

\title{Enumerating iterated tilted algebras in type $A$}

\author{Alexander E. Black}
\address{Department of Mathematics, Bowdoin College, 8600 College Station Brunswick, Maine 04011, United States of America}
\email{a.black@bowdoin.edu}
\author{Jonathan E. Gordon}
\address{Department of Mathematics, Bowdoin College, 8600 College Station Brunswick, Maine 04011, United States of America}\email{jonathan.gordon@duke.edu}
\author{Ray Maresca}
\address{Department of Mathematics, Bowdoin College, 8600 College Station Brunswick, Maine 04011, United States of America}
\email{r.maresca@bowdoin.edu}



\begin{document}

\begin{abstract}
 We show that isoclasses of iterated tilted algebras in type $A_n$ are in bijection with non-crossing spanning trees up to rotation on a convex $n+1$-gon. This is done by constructing a relationship between iterated tilted algebras up to isomorphism and exceptional sets up to isomorphic Hom-Ext quiver. 
\end{abstract}

\maketitle

\tableofcontents

\section*{Introduction}\label{sec0}
Tilting theory is a crucial topic in representation theory with its origins extending back to the late 1970's. The concept of iterated (generalized) tilted algebras was introduced by Assem and Happel in \cite{assem1981generalized}. Shortly after their paper, several classification results for iterated tilted algebras in different types arose in the literature \cite{assem1983iterated, assem1987iterated, happel1987iterated}. Moreover, the classification result of Assem and Happel led to the definition of gentle algebras and string algebras, which is now a highly active area of research. In that work, they classified the iterated tilted algebras in type $A_n$ using a combinatorial description in terms of `gentle tree quivers' on $n$ vertices. To our knowledge, there does not exist a count for how many gentle tree quivers there are on $n$ vertices in the literature. Our goal in this paper is to provide such a count as well as multiple combinatorial and algebraic interpretations. To do so, we construct bijections using the framework in Figure \ref{fig: the bijection}.

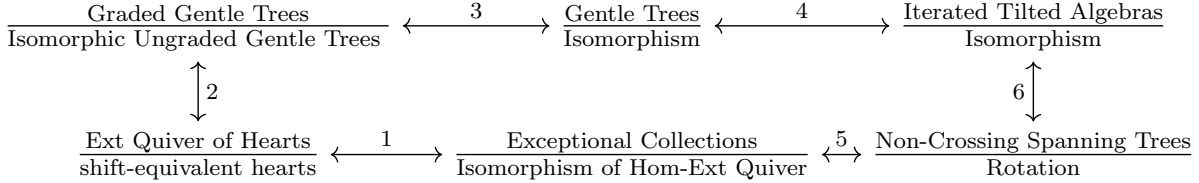
\begin{figure}[h]
\begin{center}
\begin{adjustbox}{width=\textwidth}
\begin{tikzcd}[cramped, column sep=small]
	{\frac{\text{Graded Gentle Trees}}{\text{Isomorphic Ungraded Gentle Trees } }} & {\frac{\text{Gentle Trees}}{\text{Isomorphism}}} & {\frac{\text{Iterated Tilted Algebras}}{\text{Isomorphism}}} \\
	{\frac{\text{Ext Quiver of Hearts}}{\text{shift-equivalent hearts}}} & {\frac{\text{Exceptional Collections}}{\text{Isomorphism of Hom-Ext Quiver}}} & {\frac{\text{Non-Crossing Spanning Trees}}{\text{Rotation}}}
	\arrow[r,<->, "3", from=1-1, to=1-2]
    \arrow[r,<->, "4", from=1-2, to=1-3]
    \arrow[r,<->, "2", from=1-1, to=2-1]
	\arrow[r, <->, "1", from=2-1, to=2-2]
	\arrow[r, <->, "5", from=2-2, to=2-3]
	\arrow[r, <->, "6", from=2-3, to=1-3]
    \end{tikzcd}
\end{adjustbox}
    \caption{The bijection between isoclasses of iterated tilted algebras in type $A_n$ and non-crossing spanning trees on a convex $n+1$-gon}
    \label{fig: the bijection}
\end{center}
\end{figure}

We begin by establishing bijection 1 in Figure \ref{fig: the bijection} (Lemma \ref{lem: ec and hearts}). This is a bijection between isoclasses of Ext quivers of shift-equivalent hearts in $\cD$ \textcolor{black}{and} isoclasses of Hom-Ext quivers of exceptional collections in mod-$\Lambda$. It relies on a bijection between simple-minded collections up to shift, algebraic $t$-structures up to shift, and exceptional collections of modules written down in \cite{maresca2024exceptional}. Bijection 2 (Lemma \ref{lem: hearts and graded trees}), follows quickly from Qiu's relationship (Theorem 2.11 in \cite{qiu2013ext}) between Ext quivers of hearts in $\cD := D^b(\text{mod} \Lambda)$ and associated quivers of graded gentle trees with $n$ vertices. Bijection 3 (Lemma \ref{lem: graded trees and gentle trees}) follows essentially by forgetting the grading on the graded gentle trees. 

To establish bijection 4, we first need to understand what we mean by a gentle tree. This is a quiver $Q$ with a certain 2-coloring such that each vertex has at most one arrow of each color incoming or outgoing. Proposition 2.4 in \cite{qiu2013ext}, states that $\Bbbk Q/I$ is a `gentle tree algebra' if and only if $Q$ is a gentle tree and $I$ is either $I^+$, the ideal generated by all unicolor paths of length two, or $I^-$, the ideal of all alternating color paths of length two. To construct a bijection between gentle trees up to isomorphism and iterated tilted algebras (gentle tree algebras) up to isomorphism, we first argue that given a gentle tree algebra $A$, there is a gentle tree $Q$ with a two coloring such that $A = \Bbbk Q/I^+$ (Lemma \ref{lem: canonical ideal}). Moreover, this gentle tree is unique up to isomorphism since $A$, a gentle tree algebra, is uniquely determined by $Q$ and $I$. Using this information, we establish a bijection between exceptional collections of modules up to isomorphic Hom-Ext quiver and iterated tilted algebras up to isomorphism (Theorem \ref{thm: bij bw H-E and ITA}).

In Section \ref{sec3}, we use Araya's combinatorial realization of exceptional collections in type $A_n$ and non-crossing spanning trees on a convex $n+1$-gon. In particular, we show that exceptional collections in type $A_n$ up to isomorphic Hom-Ext quiver are in bijection with non-crossing spanning trees on a convex $(n+1)$-gon up to rotation (Theorem~\ref{thm:main}), which is bijection 5 in Figure~\ref{fig: the bijection}. Combining this with the bijections from Section \ref{sec2}, we obtain Theorem~\ref{thm: bij bw ITA and trees up to rotation} which is bijection 6 in Figure \ref{fig: the bijection}. Further, in this theorem, we use the known count of non-crossing spanning trees up to rotation to obtain a formula for the number of isoclasses of iterated tilted algebras in type $A_n$. While the bijection itself is a useful piece of knowledge to have in representation theory, the relationship between iterated tilted algebras and Ext algebras of exceptional sets is quite interesting and we wonder if this type of result holds more generally for Dynkin quivers. We finish this section with a count of the number of exceptional sequences of modules in mod-$\Lambda$ where $\Lambda = \Bbbk Q$ for $Q$ a quiver of type $A_n$. We are aware that this was done for all Dynkin quivers in \cite{obaid2013number}; however here, we will provide a formula in terms of the cardinality of automorphism groups of Hom-Ext quivers, the number of linear extensions of Hom-Ext quivers, and the Coxeter number of $\Lambda$.

\section{Basic definitions and preliminaries}\label{sec1}

Throughout this paper, unless otherwise specified, let $Q$ denote a finite acyclic quiver. By mod-$\Lambda$ we denote the category of finitely generated $\Lambda:=\kk Q$ modules where $\kk$ is an algebraically closed field. Note that we assume throughout that $\Lambda$ is a hereditary algebra. By $\cD$ we mean the bounded derived category of mod-$\Lambda$, which is a triangulated category with shift functor $[1]:\cD \rightarrow \cD$. 

\subsection{Exceptional sets and combinatorial models}\label{subsec: geom models}

An indecomposable $\Lambda$-module $E$ is \textbf{exceptional} if Ext$(E,E) := $ Ext$^1(E,E) = 0$. An \textbf{exceptional sequence} is a sequence of exceptional modules $(E_1, E_2, \dots, E_k)$ such that Hom$_{\text{mod-}\Lambda}(E_i,E_j) = 0 = \text{Ext}(E_i,E_j)$ for all $i>j$. By an $\textbf{exceptional collection }   (\textbf{exceptional set})$ we mean an unordered set $E^*=\{E_1,E_2,$ $ \dots,E_k\}$ that can be ordered into an exceptional sequence in at least one way. We call an exceptional sequence (set) \textbf{complete} if $k$ is the number of simple $\Lambda$ modules. Throughout this paper, unless otherwise stated, we take all exceptional collections/sequences to be complete.

For quivers of type $A_n$, there are several combinatorial models. We will use Araya's non-crossing tree model \cite{araya2013exceptional}. Since all orientations of a type $A_n$ quiver give derived-equivalent module categories over their path algebras, it is enough to fix one orientation. We use the linear orientation
\[
Q:\quad 1 \longrightarrow 2 \longrightarrow \cdots \longrightarrow n.
\]
Let $\Lambda := \kk Q$ be its path algebra and $\cD$ the bounded derived category of mod-$\Lambda$. Since $\Lambda$ is a hereditary algebra whose underlying graph is a Dynkin diagram of type $A_n$, by Gabriel's theorem \cite{gabriel1972unzerlegbare}, the indecomposable $\Lambda$-modules are in bijection with the positive roots of the corresponding root system. These indecomposables are the \textbf{interval modules}
\[
X_{i,j}, \qquad 0 \le i < j \le n,
\]
where $X_{i,j}$ is the representation that assigns $\kk$ to each vertex $k$ with $i < k \le j$, assigns $0$ to all other vertices, and assigns the identity map to each arrow between two nonzero vertices. The derived category $\cD$, up to equivalence, does not depend on the orientation of $Q$, and every indecomposable object of $\cD$ is of the form $X_{i,j}[\ell]$ for a unique pair $0 \le i < j \le n$ and a unique integer $\ell$.

Following \cite{araya2013exceptional}, fix a regular $(n+1)$-gon with vertices labeled $0, 1, \dots, n$ counterclockwise. For $0 \le i < j \le n$, let $c(X_{i,j}) = c_{i,j}$ denote the \textbf{chord} joining vertex $i$ to vertex $j$. This assignment $X_{i,j} \mapsto c_{i,j}$ defines a bijection between the indecomposable $\Lambda$-modules and the chords of the $(n+1)$-gon. We distinguish two types of chords. A chord $c_{i,j}$ is called a \textbf{boundary chord} if its endpoints are adjacent vertices of the $(n+1)$-gon, so that $c_{i,j}$ lies on the boundary of the polygon. Otherwise $c_{i,j}$ is called an \textbf{interior chord}.

\begin{figure}[h]
\centering
\begin{tikzpicture}[scale=1, every node/.style={font=\sffamily\bfseries}]
\tikzset{
    square/.style={thick},
    redline/.style={darkred, ultra thick},
    labelnode/.style={inner sep=3pt}
}

\node[labelnode] at (-0.3, 2.3) {3};
\node[labelnode] at (2.3, 2.3) {0};
\node[labelnode] at (2.3, -0.3) {1};
\node[labelnode] at (-0.3, -0.3) {2};
\draw[square] (0,0) rectangle (2,2);
\draw[redline] (2,2) -- (0,0);
\node[redline, above left] at (1,1) {$c_{0,2}$};

\node at (5, 1) {$X_{0,2} \;=\;  \kk \longrightarrow \kk \longrightarrow 0$};

\end{tikzpicture}
\caption{An indecomposable interval module $X_{0,2}$ and its associated chord $c_{0,2}$ in the $4$-gon.}
\label{fig: interval module}
\end{figure}
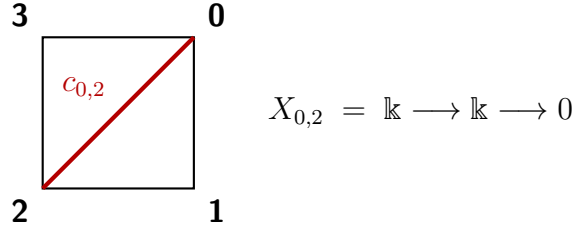

Under this bijection, exceptional pairs admit the following geometric characterization \cite[Lemma~3.2]{araya2013exceptional}: for indecomposable modules $X = X_{i,j}$ and $Y = X_{i',j'}$ with associated chords $c(X)$ and $c(Y)$,
\begin{enumerate}[label=(\arabic*)]
    \item if $c(X)$ and $c(Y)$ are disjoint, then both $(X,Y)$ and $(Y,X)$ are exceptional pairs;
    \item if $c(X)$ and $c(Y)$ cross in the interior of the polygon, then neither $(X,Y)$ nor $(Y,X)$ is an exceptional pair;
    \item if $c(X)$ and $c(Y)$ share exactly one endpoint, then exactly one of $(X,Y)$, $(Y,X)$ is an exceptional pair.
\end{enumerate}

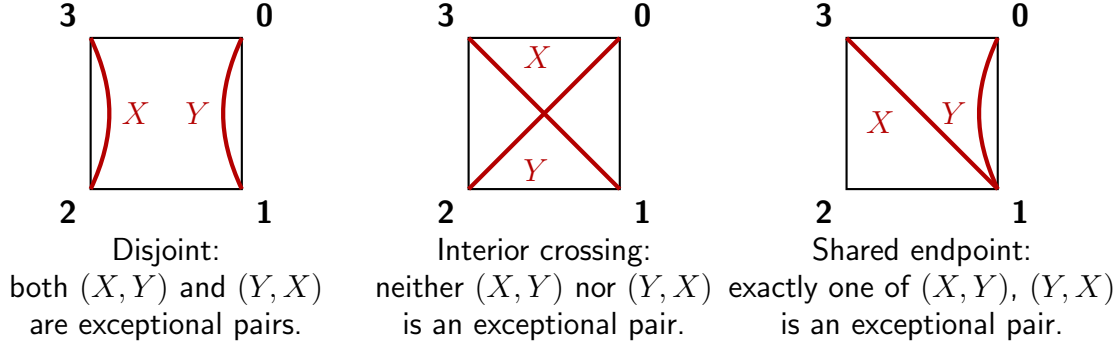
\begin{figure}[h]
\centering
\begin{tikzpicture}[scale=1, every node/.style={font=\sffamily\bfseries}]
\tikzset{
    square/.style={thick},
    redline/.style={darkred, ultra thick},
    labelnode/.style={inner sep=3pt}
}

\begin{scope}[shift={(0,0)}]
    \draw[square] (0,0) rectangle (2,2);
    \node[labelnode] at (-0.3, 2.3) {3};
    \node[labelnode] at (2.3, 2.3) {0};
    \node[labelnode] at (2.3, -0.3) {1};
    \node[labelnode] at (-0.3, -0.3) {2};
    \draw[redline] (0,2) to[bend left=25] node[right] {$X$} (0,0);
    \draw[redline] (2,0) to[bend left=25] node[left] {$Y$} (2,2);
    \node[below=0.5cm, align=center, font=\sffamily] at (1,0) {Disjoint: \\ both $(X,Y)$ and $(Y,X)$ \\ are exceptional pairs.};
\end{scope}

\begin{scope}[shift={(5,0)}]
    \draw[square] (0,0) rectangle (2,2);
    \node[labelnode] at (-0.3, 2.3) {3};
    \node[labelnode] at (2.3, 2.3) {0};
    \node[labelnode] at (2.3, -0.3) {1};
    \node[labelnode] at (-0.3, -0.3) {2};
    \draw[redline] (0,2) -- node[above right, pos=0.28] {$X$} (2,0);
    \draw[redline] (2,2) -- node[below right, pos=0.72] {$Y$} (0,0);
    \node[below=0.5cm, align=center, font=\sffamily] at (1,0) {Interior crossing: \\ neither $(X,Y)$ nor $(Y,X)$ \\ is an exceptional pair.};
\end{scope}

\begin{scope}[shift={(10,0)}]
    \draw[square] (0,0) rectangle (2,2);
    \node[labelnode] at (-0.3, 2.3) {3};
    \node[labelnode] at (2.3, 2.3) {0};
    \node[labelnode] at (2.3, -0.3) {1};
    \node[labelnode] at (-0.3, -0.3) {2};
    \draw[redline] (0,2) -- node[below left, pos=0.4] {$X$} (2,0);
    \draw[redline] (2,0) to[bend left=25] node[left] {$Y$} (2,2);
    \node[below=0.5cm, align=center, font=\sffamily] at (1,0) {Shared endpoint: \\ exactly one of $(X,Y)$, $(Y,X)$ \\ is an exceptional pair.};
\end{scope}

\end{tikzpicture}
\caption{The three configurations of chords governing exceptional pairs.}
\label{fig: three configurations}
\end{figure}

A \textbf{non-crossing spanning tree} of the $(n+1)$-gon is a set of $n$ chords that form a spanning tree of the vertex set $\{0,1,\dots,n\}$ and that pairwise do not cross in the interior of the polygon. Extending the bijection above to complete exceptional collections, we have the following result from \cite{araya2013exceptional}.

\begin{thm}[\cite{araya2013exceptional}, Theorem 2.5]
    A set of modules $\chi = \{E_1, \dots, E_n\}$ is a complete exceptional collection in mod-$\Lambda$ if and only if $\{c(E_1), \dots, c(E_n)\}$ is a non-crossing spanning tree of the $(n+1)$-gon.
    \label{thm: araya}
\end{thm} 

\subsection{$t$-structures and the bounded derived category}

A pair of summand-closed full subcategories $(A,B)$ where $A,B \subset \cD$ is called a \textbf{$t$-structure} on $\cD$ if

\begin{enumerate}
\item Hom$_{\cD}(A,B[-1])=0$   
\item $\cD = A \star B[-1]$
\item $A[1] \subset A$
\end{enumerate}

where $A\star B$ is the set of all objects $E\in\cD$ such that there is a distinguished triangle $X \rightarrow E \rightarrow Y \rightarrow X[1]$ with $X \in A$ and $Y\in B$. We call a $t$-structure \textbf{bounded} if $$\cD = \bigcup_{i\in\mathbb{Z}} A[i] = \bigcup_{i\in\mathbb{Z}} B[i].$$ The \textbf{heart} of a $t$-structure $(A,B)$ is the full subcategory $A\cap B$. A bounded $t$-structure is \textbf{algebraic} if its heart is a length category with finitely many simple objects.

\begin{rem}
    In type $A$, the heart of any bounded $t$-structure is a length category with finitely many simples, and thus is algebraic.
    \label{rem: type A algebraic}
\end{rem}

\subsection{Simple-minded collections, exceptional collections, and hearts}

A collection of objects $S$ in $\cD$ is called a $\textbf{simple-minded collection}$ if 

\begin{enumerate}
    \item dim Hom$(X,Y) = \delta_{XY}$ for all $X,Y \in S$ where $\delta_{XY}$ is the Kronecker delta.
    \item Hom$_{\cD}(X[i],Y) = 0$ for all $i\geq1$ and $X,Y\in S$.
    \item $\cD = \text{thick}_{\cD}(S)$, in our case this is equivalent to $S$ having $n$ objects where $n$ is the number of simple $\Lambda$ modules.
\end{enumerate}

There is a beautiful correspondence between simple-minded collections and hearts of algebraic $t$-structures.

\begin{thm}[\cite{koenig2014silting}, Theorem 6.1]\label{thm: koenig-yang}
    For any finite dimensional algebra $\Lambda$ over an algebraically closed field $\kk$, the following are in bijection:
    \begin{itemize}
        \item equivalence classes of silting objects in $\cK^b(\text{proj-}\Lambda)$.
        \item equivalence classes of simple-minded collections in $\cD$.
        \item algebraic $t$-structures on $\cD$.
        \item bounded co-$t$-structures of $\cK^b(\text{proj-}\Lambda)$.
    \end{itemize}
    where two sets of objects are equivalent if they additively generate the same subcategory.
\end{thm}

We note that two simple-minded collections $S_1$ and $S_2$ additively generate the same subcategory if and only if up to permutation, the objects in $S_1$ and $S_2$ are pairwise isomorphic. When studying simple-minded collections in this paper, we will fix representatives of the isoclass of each object, so we will not consider equivalence classes. Moreover, we will only need the bijection between simple-minded collections and algebraic $t$-structures in our paper, so we will not define the other terms in the bijection.

As in \cite{maresca2024exceptional}, we can insert exceptional collections into this bijection. \commentout{An object $E \in \cD$ is \textbf{exceptional} if Hom$_{\cD}(E,E[i])=0$ for all $i\in\mathbb{Z}-\{0\}$ and Hom$_{\cD}(E,E[i])\cong \kk$. A sequence of exceptional objects $(E_1, E_2, \dots, E_k)$ in $\cD$ is an \textbf{exceptional sequence} if Hom$_{\cD}(E_i,E_j[k]) = 0$ for all $i>j$ and $k\in\mathbb{Z}$. A set of exceptional objects $\{E_1, E_2, \dots, E_k\}$ in $\cD$ is an \textbf{exceptional collection } (\textbf{exceptional set}) if it can be ordered into an exceptional sequence in at least one way. An exceptional sequence (set) is \textbf{complete} if $k$ is the number of simples in mod-$\Lambda$.} Let $\cS$ be the class of simple-minded collections in $\cD$ and say that $S_1, S_2\in \cS$ are  \textbf{shift-equivalent} if and only if the unshifted objects of $S_1$ and $S_2$ coincide. We say two algebraic $t$-structures are \textbf{shift-equivalent} if and only if the simple objects in their hearts are equivalent as simple-minded collections.

\begin{thm}[\cite{maresca2024exceptional}]
    There is a bijection between
    \begin{itemize}
        \item shift-equivalence classes of simple-minded collections in $\cD$.
        \item exceptional collections in mod-$\Lambda$.
        \item shift-equivalence classes of algebraic $t$-structures on $\cD$.
        
    \end{itemize}
    \label{thm: mar24}
\end{thm}

\subsection{Iterated tilted algebras, colored quivers, and gentle trees}

Let $\Lambda = \kk Q/I$. We say that $\Lambda$ is a \textbf{gentle tree algebra} if 
\begin{itemize}
    \item $Q$ is a tree.
    \item For any vertex $i$ in $Q$, $i$ is the source and target of at most two arrows.
    \item For each arrow $\alpha$ in $Q$, there is at most one arrow $\beta$ and one arrow $\gamma$ such that $\alpha\beta \notin I$ and $\gamma\alpha\notin I$.
    \item For each arrow $\alpha$ in $Q$, there is at most one arrow $\delta$ and one arrow $\epsilon$ such that $\alpha\delta \in I$ and $\epsilon\alpha\in I$. Moreover, $I$ is generated by these monomials.
\end{itemize}

By a \textbf{tilting module} in mod-$\Lambda$, we mean a module $M$ such that 
\begin{itemize}
    \item There is an exact sequence $0\rightarrow P'' \rightarrow P' \rightarrow M \rightarrow 0$ with $P'$ and $P''$ projective modules.
    \item Ext$^1(M,M) = 0$.
    \item There is an exact sequence $0\rightarrow \Lambda \rightarrow M' \rightarrow M'' \rightarrow 0$ with $M'$ and $M''$ direct summands of $M$.
\end{itemize}

When $\Lambda$ is hereditary, a tilting module in mod-$\Lambda$ is a basic module $M = M_1 \oplus \cdots \oplus M_n$ where $n$ is the number of simple $\Lambda$ modules and Ext$(M,M) = 0$. Given a tilting $\Lambda$ module $M$, we call $A = \text{End}(M)$ a \textbf{tilted algebra}. An algebra $B$ is an \textbf{iterated tilted algebra} of $\Lambda$ if there is a finite sequence of modules $M_1, M_2, \dots, M_k$ such that $M_1$ is a tilting $A_1 := \Lambda$ module, $M_i$ is a tilting $A_{i} := \text{End}(M_{i-1})$ module, and $B = A_{k+1}$. The following classification of iterated tilted algebras in type $A_n$ was done by Assem and Happel.

\begin{thm}[\cite{assem1981generalized} and Theorem 2.2 in \cite{qiu2013ext}] \label{thm: assem happel classification}
    Let $A = \kk Q/I$. Then $A$ is an iterated tilted algebra of type $A_n$ if and only if $A$ is a gentle tree algebra with $n$ simple modules.
\end{thm}

Suppose that $\Lambda = \kk Q/I$ is a gentle tree algebra. Then we can provide a 2-coloring as in \cite{qiu2013ext}; namely, two arrows $\alpha$ and $\beta$ that are composable are colored the same if $\alpha\beta \in I$ or $\beta\alpha\in I$. Using this, we can define a \textbf{gentle tree} as a quiver $Q$ with a 2-coloring such that each vertex has at most one arrow of each color incoming or outgoing. In this paper, we consider all gentle trees up to color swap.

\begin{prop}[Proposition 2.4 in \cite{qiu2013ext}]\label{prop: coloring classification}
    The algebra $\kk Q/I$ is a gentle tree algebra if and only if $Q$ is a gentle tree and $I$ is either $I^+$, the ideal generated by all unicolor paths of length two, or $I^-$, the ideal of all alternating color paths of length two.
\end{prop}

Combining Proposition \ref{prop: coloring classification} and Theorem \ref{thm: assem happel classification}, we see that $\kk Q/I$ is an iterated tilted algebra of type $A_n$ if and only if $Q$ is a gentle tree and $I$ is either $I^+$ or $I^-$.

\subsection{Hom-Ext quivers and Ext quivers of hearts}

For any collection of modules in mod-$\Lambda$, by the \textbf{Hom-Ext quiver}, we mean the graded quiver with relations associated to the Yoneda Ext algebra where we consider Ext$^i(M,N)$ for $i = 0,1$. For a more detailed exposition, see Section 3 of \cite{igusa2025hom}. We say two Hom-Ext quivers are \textbf{isomorphic} if their underlying ungraded quivers with relations are isomorphic. Note that this definition is not as general as the one in \cite{igusa2025hom}; however, this one suffices for our needs as we only deal with quivers of type $A_n$. Now let $\cH$ be the heart of a bounded $t$-structure on $\cD$ and let $\cS$ be the finite set of simple objects in $\cH$. The \textbf{Ext quiver} of $\cH$ is the positively graded quiver with vertices the objects of $\cS$ and whose graded edges correspond to a basis of $\bigoplus_{i \in \mathbb{Z}}[-i]\text{Hom}_{\cD}(\cS,\cS[i])$.

By a \textbf{graded gentle tree} $G$, we mean a gentle tree with a positive grading on each arrow. To each graded gentle tree, we define the \textbf{associated quiver} $Q(G)$ to be the graded quiver with the same vertex set as $G$, and an arrow $a:i\rightarrow j$ for each unicolored path $p:i\rightarrow j$ in $G$, with the grading of $p$. The following result was proven in \cite{qiu2013ext}.

\begin{thm}[Theorem 2.11 in \cite{qiu2013ext}] \label{thm: Ext hearts quiver of gentle tree}
The Ext quivers of hearts in $\cD$ are precisely the associated quivers of graded gentle trees with $n$ vertices.
\end{thm}

\section{Gentle Trees}
\label{sec: gentle trees}

That $\kk Q/I$ can arise either from the $I_{+}$ ideal generated by length two paths of the same color or $I_{-}$ ideal generated by length two paths of distinct colors leads to ambiguity in the right choice of representative for $I$. In particular, it is not clear whether $I$ can always be assumed to arise from $I_{+}$, or if some choices of $I$ only appear as an $I_{-}$. In the following, we give a combinatorial argument that one can always find a representative that is of the form of $I_{+}$ by constructing an involution swapping $I_{+}$ and $I_{-}$. 

The definition previously in the literature of a gentle tree is a $2$-colored directed graph satisfying: 
\begin{itemize}
    \item[(i)] Its underlying graph is a tree
    \item[(ii)] Every vertex has in-degree at most $2$ and out-degree at most $2$
    \item[(iii)] Its edges are $2$-colored so that any two edges ending at or starting at the same vertex have distinct colors.
\end{itemize}
For our purposes, we need to disambiguate the graph from its colorings. To avoid collisions with terminology in the literature, we call a directed graph satisfying (i) and (ii) a \textbf{gentle directed tree} and a coloring of the edges of that graph satisfying (iii) a \textbf{gentle coloring}.

\begin{lem}
\label{lem:epochs}
Let $T = (V,E)$ be a gentle directed tree. Then there exists a labeling of the edges $f: E \to \{-1,1\}$ such that for all $u,v,w \in V$,
\begin{align*}
    f(u,v) &\neq f(v,w) \text{ if } (u,v), (v,w) \in E \\
    f(u,v) &= f(u,w) \text{ if } (u,v), (u,w) \in E \\
    f(u,v) &= f(w,v) \text{ if } (u,v), (w,v) \in E.
\end{align*} 
\end{lem}

\begin{proof}
Clearly this is true for a gentle directed tree with no edges. Proceed by induction. Suppose for the sake of induction that for any gentle directed tree with $1 \leq k < n$ vertices that such a choice of $f$ exists. 

Let $T$ be a gentle directed tree with $n$ vertices. Then $T$ is a tree and thus has a leaf with edge $(u,v)$ such that $u$ or $v$ has total degree $1$. By induction, there is a valid labeling with $(u,v)$ removed. Let $f$ be that labeling. Suppose first that $u$ is the vertex with total degree $1$. 

Adding $(u,v)$ back, we can force the label of $(u,v)$ to be distinct from the label of $(v,w)$ if $v$ has any out-degree. If there are two out-going edges from $v$, by assumption, they must have the same label, so the label of $(u,v)$ would still be distinct from them. Then if there is an incoming edge at $v$, $(u',v)$, it would have opposite label of the outgoing edges and therefore the same label as $(u,v)$ as desired. Hence, this is a valid labeling in that case. If any edges are missing, the scenario is only less restrictive and so a labeling can always be found. A similar argument works if $v$ is the vertex with total degree $1$. In particular, there is a symmetry of gentle directed trees by swapping the orientation on every edge, and then the induction is precisely the same.
\end{proof}

\begin{lem}\label{lem: canonical ideal}
Let $T = (V,E)$ be a gentle tree with gentle coloring $\varphi: E \to \{-1,1\}$. Then there is another gentle coloring $\psi: E \to \{-1,1\}$ such that $(T, \psi)$ is a gentle tree and for any $u, v, w \in V$ such that $(u,v), (v,w) \in E$, 
\[\psi(u,v) = \psi(v,w) \text{ if and only if } \varphi(u,v) \neq \varphi(v,w).\]
\end{lem}

\begin{proof}
By Lemma \ref{lem:epochs}, there exists $f: E \to \{-1,1\}$ such that for any $u, v, w \in V$
\begin{align*}
    f(u,v) &\neq f(v,w) \text{ if } (u,v), (v,w) \in E \\
    f(u,v) &= f(u,w) \text{ if } (u,v), (u,w) \in E \\
    f(u,v) &= f(w,v) \text{ if } (u,v), (w,v) \in E.
    \end{align*}
Define 
\[\psi(u,v) = f(u,v)\varphi(u,v).\]
First, we verify that this is still a gentle coloring. By construction of $f$ and the definition of a gentle coloring, for any $u, v, w \in V$ such that $(u,v), (u,w) \in E$, $f(u,v) = f(u,w)$ and $\varphi(u,v) \neq \varphi(u,w)$. It follows that
\[\psi(u,v) = f(u,v) \varphi(u,v) = f(u,w)\varphi(u,v) \neq f(u,w) \varphi(u,w) = \psi(u,w). \]
Thus, $\psi(u,v) \neq \psi(u,w)$. By precisely the same reasoning, for any $u,v,w \in V$ such that $(u,v)$ and $(w,v)$ are edges, $\psi(u,v) \neq \psi(w,v)$. Therefore, by definition, $\psi$ is a gentle coloring.

In addition, for any $u,v,w \in V$ such that $(u,v), (v,w) \in E$, $f(u,v) \neq f(v,w)$ or equivalently $f(u,v) = -f(v,w)$, so by definition
\[\psi(u,v) = f(u,v)\varphi(u,v) = -f(v,w) \varphi(u,v) = f(v,w) \varphi(v,w) = \psi(v,w).\]
if and only if $\varphi(u,v) = -\varphi(v,w)$, which occurs if and only if $\varphi(u,v) \neq \varphi(v,w)$ as desired.
\end{proof}

\section{Counting iterated tilted algebras}\label{sec2}

Recall the notation from Section~\ref{sec1}: let $Q$ be a quiver of type $A_n$ with linear orientation, $\Lambda = \kk Q$, and $\cD$ be the bounded derived category of mod-$\Lambda$. We will provide a count for the number of isoclasses of iterated tilted algebras of $\Lambda$ by showing that they are in bijection with isoclasses of Hom-Ext quivers of exceptional collections in mod-$\Lambda$. To do this, we first recall from Theorem \ref{thm: mar24} that for a finite-dimensional hereditary algebra $\Lambda$, the shift-equivalence classes of simple-minded collections in $D^b(\text{mod-}\Lambda)$ and the shift-equivalence classes of algebraic $t$-structures are both in bijection with the complete exceptional collections in mod-$\Lambda$ \cite[Theorem 4.4]{maresca2024exceptional}. We call two simple-minded collections \textbf{shift-equivalent} if their unshifted modules are isomorphic.

We now define an equivalence relation on exceptional collections by putting $E^*_1 \sim E^*_2$ if and only if their Hom-Ext quivers are isomorphic: $(Q^{E^*_1},I_1)\cong (Q^{E^*_2},I_2)$. We will call these \textbf{isoclasses of Hom-Ext quivers}. Similarly, we define an equivalence relation on shift-equivalence classes of simple-minded collections by saying that for any representatives $S_1$ and $S_2$ of two shift-equivalence classes, $S_1\sim S_2$ if and only if $\overline{Q(S_1)} \cong \overline{Q(S_2)}$, that is, their ungraded Ext quivers are isomorphic. Note that this is independent of choice of representative because if two simple-minded collections are shift-equivalent, their ungraded Ext quivers will be isomorphic. We will call these \textbf{isoclasses of shift-equivalent Ext quivers of simple-minded collections}. This allows us to define \textbf{isoclasses of Ext quivers of shift-equivalent hearts} as the sets of hearts of $t$-structures on $\cD$ whose simples have isomorphic ungraded Ext quivers.

\begin{lem}\label{lem: ec and hearts}
    There is a bijection between isoclasses of Ext quivers of shift-equivalent hearts in $\cD$ and isoclasses of Hom-Ext quivers of exceptional collections in mod-$\Lambda$. 
\end{lem}

\begin{proof}
    It follows from the Koenig-Yang correspondences (Theorem \ref{thm: koenig-yang}) that the simples in the hearts form a simple-minded collection, so the isoclasses of Ext quivers of shift-equivalent hearts in $\cD$ are in bijection with isoclasses of Ext quivers of shift-equivalent simple-minded collections. Now, let $S$ be a simple-minded collection and $\overline{Q(S)}$ its ungraded Ext quiver. Consider the exceptional collection $\overline{S}$ in mod-$\Lambda$ (Lemma 5.1 in \cite{simoes2022functorially} or Lemma 2.3 in 
    \cite{buan2012m}), that is, the collection of unshifted modules, and its Hom-Ext quiver $(Q^{\overline{S}},I)$. Thus, $\overline{Q(S)}$ is the ungraded Ext quiver of the exceptional collection $\overline{S}$ in mod-$\Lambda$, that is, there is one vertex for each module, and the arrows correspond to bases for Ext$^0(\overline{S},\overline{S})$ and Ext$^1(\overline{S},\overline{S})$. From $(Q^{\overline{S}},I)$, we attain $\overline{Q(S)}$ by adding an arrow from the start of a path to the end of a path for each relation-avoiding path, and removing the other relations. {\color{black} We can invert this construction as follows. In $\overline{Q(S)}$, a directed triangle of the form $\xymatrix{
  X_i \ar@/_0.7pc/[rr]_{\gamma} \ar[r]^{\alpha} & X_j \ar[r]^{\beta} & X_k}$ is present precisely when $\alpha\beta$ is a nonzero morphism or extension from $X_i$ to $X_k$. If the composition is nonzero, then this spans the one dimensional Hom or Ext space between $X_i$ and $X_k$ and the arrow $\gamma$ is precisely the arrow added in the completion of $(Q^{\overline{S}},I)$. If the composition is zero and the arrow $\gamma$ is present, the arcs corresponding to these three modules will not form a tree in Araya's geometric model, contradicting Theorem \ref{thm: araya}. Thus, to return to $(Q^{\overline{S}},I)$, we remove the arrow $\gamma$ from $\overline{Q(S)}$. If there is no such directed triangle and we have the path $X_i \rightarrow X_j \rightarrow X_k$, the lack of an arrow from $X_i$ to $X_k$ in $\overline{Q(S)}$ indicates that Hom$(X_i,X_k) = 0 =$ Ext$^1(X_i,X_k)$. Thus, the path $X_i \rightarrow X_j \rightarrow X_k$ in $(Q^{\overline{S}},I)$ is a relation.} So, under this operation, two Hom-Ext quivers are sent to isomorphic Ext quivers if and only if they are isomorphic. We conclude that isoclasses of Hom-Ext quivers of exceptional sets are in bijection with isoclasses of Ext quivers of shift-equivalent simple-minded collections, which are in bijection with isoclasses of Ext quivers of shift-equivalent hearts in $\cD$.
\end{proof}

Note that there is a surjection $\phi:\{\text{graded gentle trees}\} \rightarrow \{\text{gentle trees}\}$ gotten simply by forgetting the grading. Two graded gentle trees $T_1$ and $T_2$ get sent to isomorphic gentle trees by $\phi$ if and only if $T_1$ and $T_2$ are isomorphic as ungraded gentle trees; and if this holds, we say that $T_1$ and $T_2$ are \textbf{shift-equivalent graded gentle trees}. Then we have the following lemma.

\begin{lem}\label{lem: graded trees and gentle trees}
    There is a bijection between shift-equivalent graded gentle trees and isoclasses of gentle trees. \hfill \qed
\end{lem}

Recall that to each graded gentle tree, we have defined the associated quiver, and this association is unique up to isomorphism of graded gentle tree. Let us consider \textbf{isoclasses of associated quivers} by defining two associated quivers to be isomorphic if their ungraded quivers are isomorphic.

\begin{lem}\label{lem: hearts and graded trees}
    There is a bijection between isoclasses of Ext quivers of shift-equivalent hearts in $\cD$ and shift-equivalent graded gentle trees.
\end{lem}

\begin{proof}
    By Theorem \ref{thm: Ext hearts quiver of gentle tree}, there is a bijection between Ext quivers of hearts and quivers associated to graded gentle trees. Then, there is a surjection from the collection of Ext quivers of hearts to isoclasses of quivers associated to graded gentle trees. We see that two Ext quivers are sent to the same isoclass if and only if they are isomorphic as ungraded quivers. Therefore, there is a bijection between isoclasses of Ext quivers of shift-equivalent hearts in $\cD$ and isoclasses of associated quivers of graded gentle trees, which are in bijection with shift-equivalent graded gentle trees.
\end{proof}

Putting together all the lemmas, we obtain the following bijection.

\begin{thm}\label{thm: bij bw H-E and ITA}
    Isoclasses of Hom-Ext quivers of exceptional collections in mod-$\Lambda$ are in bijection with isoclasses of iterated tilted algebras of $\Lambda$.
\end{thm}

\begin{proof}
    By Lemmas \ref{lem: ec and hearts}, \ref{lem: graded trees and gentle trees}, and \ref{lem: hearts and graded trees}, we see that isoclasses of Hom-Ext quivers are in bijection with isoclasses of gentle trees. By Proposition \ref{prop: coloring classification}, we see that $A = \kk T/I$ is an iterated tilted algebra of type $\Lambda$ if and only if $T$ is a gentle tree and $I = I^+$ or $I=I^-$, so each gentle tree generates two iterated tilted algebras a priori. By Lemma \ref{lem: canonical ideal}, we see that we can choose one such ideal, $I^+$. In particular, if $A \cong \kk T/I^-$, there exists a $T'$ such that $A \cong \kk T'/I^+$. Moreover, such a $T'$ is unique up to isomorphism because if $A_1$ and $A_2$ are gentle tree algebras, they are isomorphic if and only if their underlying quivers with relations, $(Q_1,I_1)$ and $(Q_2,I_2)$, are isomorphic. {\color{black} Suppose $A_1 \overset{\Phi}{\rightarrow} A_2$ is an algebra isomorphism. After choosing a complete set of primitive orthogonal idempotents for $A_1$ and $A_2$, this induces a bijection $\varphi$ from the chosen idempotents of $A_1$ to the chosen idempotents of $A_2$, and hence a bijection between the vertices of their respective Gabriel quivers $Q_1 \overset{\varphi}{\rightarrow} Q_2$. For any two vertices $x,y\in Q_1$, the isomorphism also induces a linear bijection $e_x\left({\text{rad}(A_1)\over \text{rad}^2(A_1)}\right)e_y \rightarrow e_{\varphi(x)}\left({\text{rad}(A_2)\over \text{rad}^2(A_2)}\right)e_{\varphi(y)}$. Since $A_1$ and $A_2$ are gentle tree algebras, both these spaces are at most one dimensional, so an arrow $x\rightarrow y$ exists in $Q_1$ if and only if an arrow $\varphi(x)\rightarrow \varphi(y)$ exists in $Q_2$ and $\varphi$ is an isomorphism between the Gabriel quivers. Suppose now that there is a length two path $i \overset{\alpha}{\rightarrow} j \overset{\beta}{\rightarrow} k$ in $Q_1$. We have $\alpha\beta \in I_1$ if and only if $e_kA_1e_i = 0$. Since $A_1 \overset{\Phi}{\cong} A_2$, this vanishing is equivalent to $e_{\varphi(k)}A_2e_{\varphi(i)} = 0$, and we conclude $\alpha\beta \in I_1$ if and only if $\varphi(\alpha)\varphi(\beta)\in I_2$. Thus, $\varphi(I_1) = I_2$ and $(Q_1,I_1) \cong (Q_2,I_2)$.} Finally, we conclude that isoclasses of iterated tilted algebras are in bijection with isoclasses of gentle trees, which are in bijection with Hom-Ext quivers of exceptional collections in mod-$\Lambda$. 
\end{proof}

\section{Hom--Ext Quivers and Non-Crossing Trees Up To Rotation}\label{sec3}

To complete the enumeration of isoclasses of iterated tilted algebras, we must count the number of isoclasses of Hom-Ext quivers of exceptional collections in mod-$\Lambda$. To do this, we will use geometric models. In this section we prove that two complete exceptional collections in type \(A_n\) have isomorphic Hom--Ext quivers if and only if their associated non-crossing spanning trees differ by a cyclic rotation of the \((n+1)\)-gon. This, combined with Theorem \ref{thm: bij bw H-E and ITA} and a count of non-crossing spanning trees up to rotation will yield a count for the number of isoclasses of iterated tilted algebras of a quiver of type $A_n$, along with a count of the number of exceptional collections in type $A_n$ up to isomorphic Hom-Ext quiver.

We continue to fix the linearly oriented quiver \(Q\) of type \(A_n\) from Section~\ref{sec1}, with path algebra \(\Lambda := \kk Q\) and indecomposable modules \(X_{i,j}\) for \(0 \le i < j \le n\). Recall Araya's geometric model \cite{araya2013exceptional} from Section \ref{subsec: geom models}, where each \(X_{i,j}\) is identified with the chord \(c_{i,j}\) of the regular \((n+1)\)-gon joining vertices \(i\) and \(j\), and under this identification, complete exceptional collections (up to permutation) correspond bijectively to non-crossing spanning trees of the \((n+1)\)-gon \cite[Theorem 2.5]{araya2013exceptional}. We write \(T(\chi)\) for the tree associated to a complete exceptional collection \(\chi\). 

Let \((Q^\chi,R^\chi)\) be the Hom--Ext quiver with relations associated to \(\chi\) as in \cite[Definitions 3.1 and 3.3]{igusa2025hom}. We will prove the following theorem.

\begin{thm}\label{thm:main} Let \(\chi\) and \(\chi'\) be complete exceptional collections in mod-\(\Lambda\). Then
\[
(Q^\chi,R^\chi) \cong (Q^{\chi'},R^{\chi'})
\]
if and only if
\[
T(\chi')=\sigma^m(T(\chi))
\]
for some cyclic rotation \(\sigma^m\) of the \((n+1)\)-gon.
\end{thm}

Our strategy is to define a function \(F\) that maps non-crossing spanning trees to quivers with relations, prove that \(F(T) \cong F(T')\) if and only if \(T\) and \(T'\) differ by a rotation (Theorem~\ref{thm:comb}), and then verify that \(F(T(\chi))\)
recovers \((Q^\chi, R^\chi)\) (Proposition~\ref{prop:F-HomExt}). Combining these yields Theorem~\ref{thm:main}.

\subsection{The function \(F\)}
\label{subsec: F}

Throughout this section, the regular \((n+1)\)-gon \(P_{n+1}\) has vertices labeled \(0, 1, \dots, n\) in clockwise order, identified with elements of \(\mathbb{Z}/(n+1)\mathbb{Z}\) so that all index arithmetic is taken modulo \(n+1\). At each polygon vertex \(p\), the local cyclic order on incident tree edges is the counterclockwise order around \(p\). All cyclic order conditions below refer to this counterclockwise order.

In this section a tree \(T\) refers to a non-crossing spanning tree in $P_{n+1}$, and is a different combinatorial object from the gentle trees of Section~\ref{sec2}.

\begin{defn}
Let \(T\) be a non-crossing spanning tree of \(P_{n+1}\). Define \(F(T)=(Q_T,R_T)\) as follows. 

\textbf{Vertices:} The vertices of \(Q_T\) are $\{e_1, e_2, \dots, e_r\}$ and are in bijection with the edges of \(T\).

\textbf{Arrows:} For each polygon vertex \(p\), let
\[
e_{i_1},e_{i_2},\dots,e_{i_r}
\]
be the edges of \(T\) incident to \(p\), listed in counterclockwise order. Add arrows
\[
e_{i_1} \to e_{i_2} \to \dots \to e_{i_r}
\]

\textbf{Relations:} Relations are generated by length two paths in which the first and last edges don't share an endpoint. More formally, for a length-two path
\[
e_i\to e_j\to e_k
\]
in \(Q_T\), let the arrow \(e_i \to e_j\) come from a polygon vertex, \(p\), and the arrow \(e_j \to e_k\) from vertex \(q\). Both \(p\) and \(q\) are endpoints of \(e_j\). Then,
\[
e_i \to e_j \to e_k=0 \qquad \text{if and only if} \qquad p \neq q.
\]
\label{def: F}
\end{defn}

Figure~\ref{fig:F-construction-A3} illustrates the function \(F\) for type \(A_3\), showing two of the four total rotation classes of non-crossing spanning trees for the square together with their associated Hom--Ext quivers.

\begin{rem}[Right-hand-rule]
Both the arrow rule and the relation rule can be remembered by a right-hand rule. Place the polygon flat with vertices labeled clockwise as the reader sees them, and orient the right hand so that its fingers curl in the counterclockwise direction around the polygon (equivalently, the thumb points out the page). Arrows go in the direction of the finger-curl around each polygon vertex. A length-two path is zero exactly when its two arrows are read off at two different polygon vertices. Thus, when following the curl at an endpoint \(f\), if the curl points to the other endpoint of \(f\), there is a relation.
\label{rem: rhr}
\end{rem}

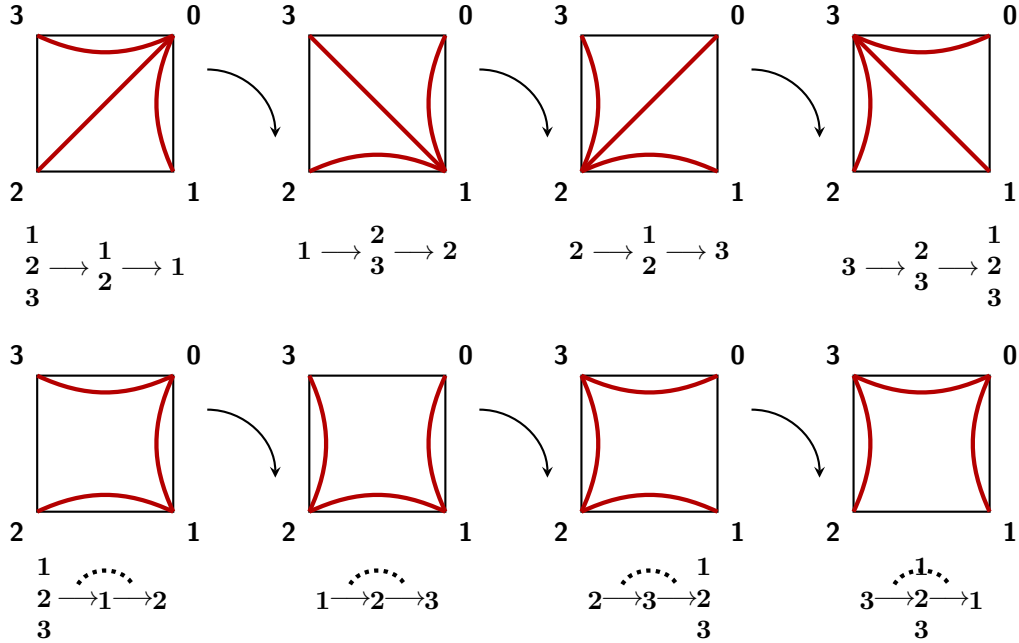
\begin{figure}[h]
\centering
\centering
\begin{tikzpicture}[scale=0.9, every node/.style={font=\sffamily\bfseries\footnotesize}]
\tikzset{
    square/.style={thick},
    redline/.style={darkred, ultra thick},
    labelnode/.style={inner sep=3pt},
    arrow/.style={->, >=stealth, ultra thick},
    transarrow/.style={->, >=stealth, thick}
}

\begin{scope}[shift={(0,0)}]
    \draw[square] (0,0) rectangle (2,2);
    \node[labelnode] at (-0.3, 2.3) {3};
    \node[labelnode] at (2.3, 2.3) {0};
    \node[labelnode] at (2.3, -0.3) {1};
    \node[labelnode] at (-0.3, -0.3) {2};
    \draw[redline] (0,0) -- (2,2);
    \draw[redline] (0,2) to[bend right=25] (2,2);
    \draw[redline] (2,0) to[bend left=25] (2,2);
    \node[below=0.5] at (1,0) {
        \(\begin{matrix} \mathbf{1} \\ \mathbf{2} \\ \mathbf{3} \end{matrix} \longrightarrow \begin{matrix} \mathbf{1} \\ \mathbf{2} \end{matrix} \longrightarrow \mathbf{1}\)
    };
    \draw[transarrow] (2.5, 1.5) to[bend left=45] (3.5, 0.5);
\end{scope}

\begin{scope}[shift={(4,0)}]
    \draw[square] (0,0) rectangle (2,2);
    \node[labelnode] at (-0.3, 2.3) {3};
    \node[labelnode] at (2.3, 2.3) {0};
    \node[labelnode] at (2.3, -0.3) {1};
    \node[labelnode] at (-0.3, -0.3) {2};
    \draw[redline] (0,2) -- (2,0);
    \draw[redline] (2,2) to[bend right=25] (2,0);
    \draw[redline] (0,0) to[bend left=25] (2,0);
    \node[below=0.5] at (1,0) {
        $\mathbf{1} \longrightarrow \begin{matrix} \mathbf{2} \\ \mathbf{3} \end{matrix} \longrightarrow \mathbf{2}$
    };
    \draw[transarrow] (2.5, 1.5) to[bend left=45] (3.5, 0.5);
\end{scope}

\begin{scope}[shift={(8,0)}]
    \draw[square] (0,0) rectangle (2,2);
    \node[labelnode] at (-0.3, 2.3) {3};
    \node[labelnode] at (2.3, 2.3) {0};
    \node[labelnode] at (2.3, -0.3) {1};
    \node[labelnode] at (-0.3, -0.3) {2};
    \draw[redline] (0,0) -- (2,2);
    \draw[redline] (0,2) to[bend left=25] (0,0);
    \draw[redline] (2,0) to[bend right=25] (0,0);
    \node[below=0.5] at (1,0) {
        $\mathbf{2} \longrightarrow \begin{matrix} \mathbf{1} \\ \mathbf{2} \end{matrix} \longrightarrow \mathbf{3}$
    };
    \draw[transarrow] (2.5, 1.5) to[bend left=45] (3.5, 0.5);
\end{scope}

\begin{scope}[shift={(12,0)}]
    \draw[square] (0,0) rectangle (2,2);
    \node[labelnode] at (-0.3, 2.3) {3};
    \node[labelnode] at (2.3, 2.3) {0};
    \node[labelnode] at (2.3, -0.3) {1};
    \node[labelnode] at (-0.3, -0.3) {2};
    \draw[redline] (0,2) -- (2,0);
    \draw[redline] (2,2) to[bend left=25] (0,2);
    \draw[redline] (0,0) to[bend right=25] (0,2);
    \node[below=0.5] at (1,0) {
        $\mathbf{3} \longrightarrow \begin{matrix} \mathbf{2} \\ \mathbf{3} \end{matrix} \longrightarrow \begin{matrix} \mathbf{1} \\ \mathbf{2} \\ \mathbf{3} \end{matrix}$
    };
\end{scope}

\begin{scope}[shift={(0,-5)}]
    \draw[square] (0,0) rectangle (2,2);
    \node[labelnode] at (-0.3, 2.3) {3};
    \node[labelnode] at (2.3, 2.3) {0};
    \node[labelnode] at (2.3, -0.3) {1};
    \node[labelnode] at (-0.3, -0.3) {2};
    \draw[redline] (0,2) to[bend right=25] (2,2);
    \draw[redline] (2,2) to[bend right=25] (2,0);
    \draw[redline] (2,0) to[bend right=25] (0,0);
    \node (n1) at (0.1, -1.3) {$\begin{matrix} \mathbf{1} \\ \mathbf{2} \\ \mathbf{3} \end{matrix}$};
    \node (n2) at (0.6, -1.3) {$\longrightarrow$};
    \node (n3) at (1.0, -1.3) {$\mathbf{1}$};
    \node (n4) at (1.4, -1.3) {$\longrightarrow$};
    \node (n5) at (1.8, -1.3) {$\mathbf{2}$};
    \draw[dotted, ultra thick] (n2.north) to[bend left=60] (n4.north);
    \draw[transarrow] (2.5, 1.5) to[bend left=45] (3.5, 0.5);
\end{scope}

\begin{scope}[shift={(4,-5)}]
    \draw[square] (0,0) rectangle (2,2);
    \node[labelnode] at (-0.3, 2.3) {3};
    \node[labelnode] at (2.3, 2.3) {0};
    \node[labelnode] at (2.3, -0.3) {1};
    \node[labelnode] at (-0.3, -0.3) {2};
    \draw[redline] (2,2) to[bend right=25] (2,0);
    \draw[redline] (2,0) to[bend right=25] (0,0);
    \draw[redline] (0,0) to[bend right=25] (0,2);
    \node (m1) at (0.2, -1.3) {$\mathbf{1}$};
    \node (m2) at (0.6, -1.3) {$\longrightarrow$};
    \node (m3) at (1.0, -1.3) {$\mathbf{2}$};
    \node (m4) at (1.4, -1.3) {$\longrightarrow$};
    \node (m5) at (1.8, -1.3) {$\mathbf{3}$};
    \draw[dotted, ultra thick] (m2.north) to[bend left=60] (m4.north);
    \draw[transarrow] (2.5, 1.5) to[bend left=45] (3.5, 0.5);
\end{scope}

\begin{scope}[shift={(8,-5)}]
    \draw[square] (0,0) rectangle (2,2);
    \node[labelnode] at (-0.3, 2.3) {3};
    \node[labelnode] at (2.3, 2.3) {0};
    \node[labelnode] at (2.3, -0.3) {1};
    \node[labelnode] at (-0.3, -0.3) {2};
    \draw[redline] (0,2) to[bend right=25] (2,2);
    \draw[redline] (2,0) to[bend right=25] (0,0);
    \draw[redline] (0,0) to[bend right=25] (0,2);
    \node (k1) at (0.2, -1.3) {$\mathbf{2}$};
    \node (k2) at (0.6, -1.3) {$\longrightarrow$};
    \node (k3) at (1.0, -1.3) {$\mathbf{3}$};
    \node (k4) at (1.4, -1.3) {$\longrightarrow$};
    \node (k5) at (1.8, -1.3) {$\begin{matrix} \mathbf{1} \\ \mathbf{2} \\ \mathbf{3} \end{matrix}$};
    \draw[dotted, ultra thick] (k2.north) to[bend left=60] (k4.north);
    \draw[transarrow] (2.5, 1.5) to[bend left=45] (3.5, 0.5);
\end{scope}

\begin{scope}[shift={(12,-5)}]
    \draw[square] (0,0) rectangle (2,2);
    \node[labelnode] at (-0.3, 2.3) {3};
    \node[labelnode] at (2.3, 2.3) {0};
    \node[labelnode] at (2.3, -0.3) {1};
    \node[labelnode] at (-0.3, -0.3) {2};
    \draw[redline] (0,2) to[bend right=25] (2,2);
    \draw[redline] (2,2) to[bend right=25] (2,0);
    \draw[redline] (0,0) to[bend right=25] (0,2);
    \node (l1) at (0.2, -1.3) {$\mathbf{3}$};
    \node (l2) at (0.6, -1.3) {$\longrightarrow$};
    \node (l3) at (1.0, -1.3) {$\begin{matrix} \mathbf{1} \\ \mathbf{2} \\ \mathbf{3} \end{matrix}$};
    \node (l4) at (1.4, -1.3) {$\longrightarrow$};
    \node (l5) at (1.8, -1.3) {$\mathbf{1}$};
    \draw[dotted, ultra thick] (l2.north) to[bend left=60] (l4.north);
\end{scope}
\end{tikzpicture}
\caption{Two rotation classes of non-crossing spanning trees of the square (type $A_3$), each shown with its corresponding Hom--Ext Quiver. The dotted arcs indicate that the corresponding path is a relation in $F(T)$.}
\label{fig:F-construction-A3}
\end{figure}

\subsection{Map to Hom--Ext Quivers}

Let \(\chi\) be a complete exceptional collection in type \(A_n\), and let \(T(\chi)\) be the corresponding non-crossing spanning tree under Araya's bijection.

\begin{prop}\label{prop:F-HomExt}
For every complete exceptional collection \(\chi\) in mod-\(\Lambda\),
\[
(Q^\chi,R^\chi) \cong F(T(\chi)).
\]
\end{prop}

\begin{proof}
By Araya's bijection, the objects of \(\chi\) are identified with the edges of the non-crossing spanning tree \(T(\chi)\). Thus the vertex sets of \((Q^\chi,R^\chi)\) and \(F(T(\chi))\) agree.

We compare arrows. Let \(e\) and \(f\) be two edges of \(T(\chi)\), corresponding to indecomposables \(X_e\) and \(X_f\).

If \(e\) and \(f\) do not share a polygon endpoint, then the corresponding chords are disjoint. By Araya's definition, defined in Section \ref{subsec: geom models}, the pair is exceptional in both orders, so there is no irreducible Hom or Ext arrow between them.

Now suppose $e$ and $f$ share a polygon endpoint $p$. If they are not consecutive at $p$, then there is an edge $g$ of $T(\chi)$ between them in the cyclic order at $p$. Write $e = c_{p,i}$, $f = c_{p,j}$, and $g = c_{p,k}$ with $k$ strictly between $i$ and $j$ counterclockwise at $p$. Using Araya's bijection we know that of the four spaces $\Hom_\Lambda(X_e,X_f)$, $\Hom_\Lambda(X_f,X_e)$, $\Ext^1_\Lambda(X_e,X_f)$ and $\Ext^1_\Lambda(X_f,X_e)$ exactly one is nonzero, and it is one-dimensional. The nonzero one factors through $X_g$, since the chord $c_{p,k}$ separates $c_{p,i}$ from $c_{p,j}$ at $p$. Since $X_g \in \chi$, this class is reducible in the Hom-Ext quiver of $\chi$, and contributes no arrow.

If \(e\) and \(f\) are consecutive at \(p\), then there is no intermediate edge at \(p\). Araya's definition gives exactly one nonzero Hom or Ext direction between the two objects, and this class cannot factor through another object of \(\chi\). Therefore it gives one irreducible arrow. With the chosen orientation of the polygon, this arrow is precisely the arrow from the earlier edge to the later edge in the counterclockwise local order at \(p\). Hence the arrows agree with those of \(F(T(\chi))\).

Now we identify the relations. By \cite{igusa2025hom}, the relations in the Hom-Ext quiver are generated by length-two compositions of irreducible arrows that vanish. Consider a length-two path
\[
e \longrightarrow f \longrightarrow g
\]
in $(Q^\chi, R^\chi)$. The arrow $e \to f$ arises at a polygon vertex $p$ shared by $e$ and $f$, and the arrow $f \to g$ arises at a polygon vertex $q$ shared by $f$ and $g$, so $p$ and $q$ are endpoints of $f$. Since $T(\chi)$ is a tree it contains no $3$-cycle, so $c_e$ and $c_g$ share a polygon endpoint if and only if $p = q$.

If $p = q$, then $e$ and $g$ are chords at $p$ separated by $f$ in the local order there. As in the discussion of arrows above, the nonzero class from $X_e$ to $X_g$ factors through $X_f$. Every class from $X_e$ to $X_f$ is a multiple of the arrow $e \to f$, and every class from $X_f$ to $X_g$ a multiple of the arrow $f \to g$, so if the composite of these two arrows were zero, no nonzero class from $X_e$ to $X_g$ could factor through $X_f$. Thus, the composite $X_e \to X_f \to X_g$ is nonzero.

If $p \neq q$, then $c_e$ and $c_g$ are disjoint, so by Araya's bijection $\Hom_\Lambda(X_e,X_g) = \Ext^1_\Lambda(X_e,X_g) = 0$. The composite $X_e \to X_f \to X_g$ must lie in one of these two spaces and hence vanishes. So the vanishing length-two paths are exactly those with $p \neq q$, which are the relations of $F(T(\chi))$ by Definition~\ref{def: F}.

Therefore
\[
(Q^\chi,R^\chi) \cong F(T(\chi)).
\]
\end{proof}

\subsection{Boundary Leaf Edges}

\begin{defn}
A \textbf{boundary leaf edge} of \(T\) is an edge
\[
s=\{i,i+1\}
\]
of \(T\) such that one endpoint of \(s\) is incident to no other edge of \(T\).
\label{defn: boundary leaf edge}
\end{defn}

\begin{lem}\label{lem:boundary-leaf-exists}
Every non-crossing spanning tree in a polygon with at least three vertices has at least two boundary leaf edges.
\end{lem}

\begin{proof}
We argue by induction on the number $n$ of polygon vertices.

A non-crossing spanning tree of a triangle consists of two of its three boundary edges. Each of these two edges has an endpoint incident to no other edge of $T$, so both are boundary leaf edges.

Suppose now that the claim holds for all polygons with fewer than $n$ vertices, and let $T$ be a non-crossing spanning tree of an $n$-gon $P$.

If $T$ consists entirely of boundary edges of $P$, then $T$ is a path along the boundary of $P$, and its two end edges are boundary leaf edges.

Otherwise, $T$ contains an internal chord $e = \{a, b\}$ of $P$. The chord $e$ divides $P$ into two sub-polygons $P_1$ and $P_2$, each having fewer than $n$ vertices and containing $e$ as a boundary edge. Since $T$ is non-crossing, the edges of $T$ contained in $P_i$, together with $e$, form a non-crossing spanning tree $T_i$ of $P_i$ for $i = 1, 2$.

By the inductive hypothesis, each $T_i$ has at least two boundary leaf edges. Note that $e$ is a boundary edge of $P_i$, so $e$ itself may not be a boundary leaf edge of $T_i$. However, at most one of the two guaranteed boundary leaf edges of $T_i$ can equal $e$, so each $T_i$ has at least one boundary leaf edge $s_i \neq e$. Each such $s_i$ is a boundary edge of the original polygon $P$. Hence $s_i$ is a boundary leaf edge of $T$.

Thus $s_1$ and $s_2$ are two distinct boundary leaf edges of $T$ and the result follows by induction.
\end{proof}

\begin{lem}\label{lem:degree-one-boundary-leaf}
Let \(s\) be an edge of \(T\). Then \(s\) is a boundary leaf edge if and only if the corresponding vertex of \(Q_T\) has degree \(1\). 
\end{lem}

\begin{proof}
Let \(s=\{p,q\}\). At an endpoint \(p\), the contribution of \(s\) to the degree of the corresponding vertex in \(Q_T\) is \(0\) if no other tree edge is incident to \(p\); \(1\) if \(s\) is first or last in the local order of edges incident to \(p\); and \(2\) if \(s\) lies between two other incident edges in that local order.

Suppose \(s\) is a boundary leaf edge. One endpoint of \(s\) is incident to no other edges, contributing \(0\). The other endpoint must be either first or last in the local cyclic order of incident tree edges, contributing \(1\). Hence \(s\) has degree \(1\).

Suppose instead that \(s\) has degree \(1\). Then one endpoint of \(s\) contributes \(0\), so that endpoint is incident to no other tree edge. Call the endpoint incident to no other tree edge \(p\), and call the other endpoint \(q\). Suppose for contradiction that \(s\) is an internal chord. The chord \(s\) separates the polygon into two open regions. Since \(p\) is incident only to \(s\), the polygon vertices in each region must connect to the rest of \(T\) through \(q\) without crossing \(s\). Hence \(q\) is incident to tree edges in both regions, so \(s\) lies between two other incident edges in the local cyclic order at \(q\). This contributes \(2\) to the degree, contradicting degree \(1\). Therefore \(s\) is a boundary edge, and since one of its endpoints is incident to no other tree edge, it is a boundary leaf edge.
\end{proof}

\begin{lem}\label{lem:boundary-leaf-detect}
Let
\[
\varphi:F(T)\to F(T')
\]
be an isomorphism of quivers with relations. Then there exists a boundary leaf edge \(s\subset T\) such that \(\varphi(s)\) is a boundary leaf edge of \(T'\).
\end{lem}

\begin{proof}
By Lemma~\ref{lem:boundary-leaf-exists}, choose a boundary leaf edge \(s\subset T\). By Lemma~\ref{lem:degree-one-boundary-leaf}, the vertex \(s\) has degree \(1\) in \(Q_T\).

Since \(\varphi\) is a quiver isomorphism, \(\varphi(s)\) has degree \(1\) in \(Q_{T'}\). By Lemma~\ref{lem:degree-one-boundary-leaf} again, \(\varphi(s)\) is a boundary leaf edge of \(T'\).
\end{proof}

\subsubsection{Deleting a Boundary Leaf Edge}

\begin{defn}
Let \(s\) be a boundary leaf edge of \(T\) and let \(x\) be the endpoint of \(s\) incident to no other edge of \(T\). Define \(T/s\) by deleting the polygon vertex \(x\) and the tree edge \(s\), then relabeling the remaining polygon vertices cyclically as in Figure~\ref{fig:deletion-reinsertion}. Then \(T/s\) is a non-crossing spanning tree of an \(n\)-gon.
\label{defn: deletion}
\end{defn}

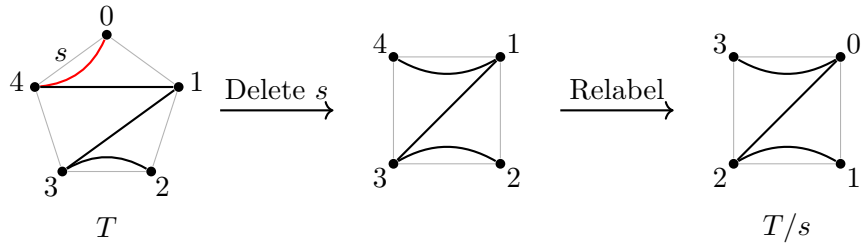
\begin{figure}[h]
\centering
\begin{tikzpicture}[scale=1.0,
  polyvert/.style={circle, fill=black, inner sep=1.3pt},
  polyedge/.style={gray!60, thin},
  treeedge/.style={thick, black},
  leafedge/.style={thick, red},
  neighboredge/.style={thick, ForestGreen},
  newedge/.style={thick, blue}]

\begin{scope}[shift={(0,0)}]
  \foreach \i in {0,...,4} {
    \node[polyvert] (T\i) at (90 - 72*\i:1) {};
    \node at (90 - 72*\i:1.25) {\small $\i$};
  }
  \foreach \i/\j in {0/1,1/2,2/3,3/4,4/0} {
    \draw[polyedge] (T\i) -- (T\j);
  }
  \draw[leafedge] (T0) to[bend left=30] node[midway, above left, black]
    {\small $s$} (T4);
  \draw[treeedge] (T4) -- (T1);
  \draw[treeedge] (T1) -- (T3);
  \draw[treeedge] (T3) to[bend left=30] (T2);
  \node at (0,-1.55) {\small $T$};
\end{scope}

\draw[->, thick] (1.5,0) -- (3,0) node[midway, above] {\small Delete $s$};

\begin{scope}[shift={(4.5,0)}]
  \node[polyvert] (S1) at (45 - 90*1:1) {};
  \node at (45 - 90*1:1.25) {\small $2$};
  \node[polyvert] (S2) at (45 - 90*2:1) {};
  \node at (45 - 90*2:1.25) {\small $3$};
  \node[polyvert] (S3) at (45 - 90*3:1) {};
  \node at (45 - 90*3:1.25) {\small $4$};
  \node[polyvert] (S4) at (45 - 90*4:1) {};
  \node at (45 - 90*4:1.25) {\small $1$};
  \foreach \i/\j in {1/2,2/3,3/4,4/1} {
    \draw[polyedge] (S\i) -- (S\j);
  }
  \draw[treeedge] (S3) to[bend right=30]  (S4);
  \draw[treeedge] (S4) -- (S2);
  \draw[treeedge] (S2) to[bend left=30] (S1);
\end{scope}

\draw[->, thick] (6,0) -- (7.5,0) node[midway, above] {\small Relabel};

\begin{scope}[shift={(9,0)}]
  \foreach \i in {0,...,3} {
    \node[polyvert] (S\i) at (45 - 90*\i:1) {};
    \node at (45 - 90*\i:1.25) {\small $\i$};
  }
  \foreach \i/\j in {0/1,1/2,2/3,3/0} {
    \draw[polyedge] (S\i) -- (S\j);
  }
  \draw[treeedge] (S3) to[bend right=30]  (S0);
  \draw[treeedge] (S0) -- (S2);
  \draw[treeedge] (S2) to[bend left=30] (S1);
  \node at (0,-1.55) {\small $T/s$};
\end{scope}

\end{tikzpicture}
\caption{Deletion of the boundary leaf edge $s$ (red) from $T$ removes the  vertex $0$, and the remaining vertices are relabeled to preserve cyclic order, giving $T/s$.}
\label{fig:deletion-reinsertion}
\end{figure}

\begin{lem}\label{lem:deletion-compatible}
Let \(s\) be a boundary leaf edge of \(T\). Then
\[
F(T/s)
\]
is obtained from \(F(T)\) by deleting the vertex \(s\) and its unique incident arrow, and by keeping all remaining arrows and relations.
\end{lem}

\begin{proof}
Let $s = \{i,j\}$, where $i$ is the endpoint incident to no other edges. The endpoint $i$ contributes no arrows to \(F(T)\). At the other endpoint $j$, the edge \(s\) is first or last in the counterclockwise order of incident tree edges, since \(s\) is a boundary leaf edge. Hence \(s\) is adjacent to exactly one other edge in the local order, so \(s\) contributes exactly one arrow.

Deleting \(s\) removes this arrow. Since \(s\) was first or last in the local order, deleting it does not make two previously nonconsecutive edges become consecutive. Thus no new arrow is created.

All other polygon vertices and all other local orders are unchanged. Therefore all remaining arrows are exactly the arrows of \(F(T/s)\).

Every relation involving \(s\) is removed. Any relation not involving \(s\) depends only on the cyclic order of the polygon vertices appearing in the corresponding length-two path. Deleting a boundary leaf edge removes one boundary vertex but does not change the cyclic order of the remaining vertices. Therefore all such relations are preserved.

Hence \(F(T/s)\) is obtained from \(F(T)\) by deleting \(s\) and its unique incident arrow.
\end{proof}

\begin{lem}\label{lem:induced-isomorphism-after-deletion}
Let
\[
\varphi:F(T)\to F(T')
\]
be an isomorphism of quivers with relations. Suppose \(s\subset T\) is a boundary leaf edge and \(s'=\varphi(s)\) is a boundary leaf edge of \(T'\). Then \(\varphi\) induces an isomorphism
\[
F(T/s)\cong F(T'/s').
\]
\end{lem}

\begin{proof}
By Lemma~\ref{lem:deletion-compatible}, \(F(T/s)\) is obtained from \(F(T)\) by deleting the vertex \(s\) and its unique incident arrow. Likewise, \(F(T'/s')\) is obtained from \(F(T')\) by deleting \(s'\) and its unique incident arrow.

Since \(\varphi(s)=s'\), the isomorphism \(\varphi\) restricts to an isomorphism on the remaining vertices and arrows. It also preserves all relations not involving \(s\), because \(\varphi\) is an isomorphism of quivers with relations.

Therefore \(\varphi\) induces
\[
F(T/s)\cong F(T'/s').
\]
\end{proof}

\subsection{Reinsertion Lemma}

\begin{defn}
Let \(S\) be a non-crossing spanning tree in an \(n\)-gon. A \textbf{boundary leaf reinsertion} of \(S\), between vertices $i$ and $i+1$, is the operation of inserting a new polygon vertex \(x\) between \(i\) and \(i+1\), and adding exactly one of the two new boundary edges \(\{i,x\}\) or \(\{x,i+1\}\) to the tree. The added edge is denoted \(s\), and the resulting tree \(T\) is again a non-crossing spanning tree.
\label{defn: reinsertion}
\end{defn}

\begin{lem}\label{lem:reinsertion-unique}
Let \(T_1, T_2\) be non-crossing spanning trees in the \((n+1)\)-gon, with boundary leaf edges \(s_1\) and \(s_2\) respectively, and suppose \(S := T_1/s_1 = T_2/s_2\) as trees in the \(n\)-gon. Let
\[
\psi: F(T_1) \to F(T_2)
\]
be an isomorphism of quivers with relations such that \(\psi(s_1) = s_2\) and the induced map on \(F(S)\) is the identity. Then \(T_1 = T_2\); in particular \(s_1 = s_2\).
\end{lem}

\begin{proof}
Let \(p_j\) be the endpoint of \(s_j\) incident to no other edge of \(T_j\) and let \(q_j\) be the other endpoint of \(s_j\). By Lemma~\ref{lem:deletion-compatible}, \(F(T_j)\) is obtained from \(F(S)\) by adjoining the vertex \(s_j\) and one further arrow \(\alpha_j\), and the relations of \(F(T_j)\) not involving \(s_j\) are exactly those of \(F(S)\). Since \(p_j\) is incident to no other edge of \(T_j\), it contributes no arrow and \(\alpha_j\) arises at \(q_j\). Since \(s_j\) must be first or last in the counterclockwise order of tree edges incident to \(q_j\), it is adjacent there to exactly one other edge, denoted as \(e_j\), and \(\alpha_j\) runs \(e_j\to s_j\) if \(s_j\) is last in the counterclockwise order at \(q_j\) and \(s_j\to e_j\) if \(s_j\) is first. In the first case \(s_j\) is a sink of \(F(T_j)\), so every length-two path through \(s_j\) has the form \(a\to e_j\to s_j\) with \(a\to e_j\) an arrow of \(F(S)\), and by Definition~\ref{def: F} such a path is a relation exactly when \(a\to e_j\) arises at the endpoint of \(e_j\) other than \(q_j\). In the second case \(s_j\) is a source and the dual statement holds.

Since \(\psi\) sends \(s_1\) to \(s_2\) and induces the identity on \(F(S)\), it sends \(\alpha_1\) to \(\alpha_2\), so \(e_1 = e_2 =: e\). There are a priori four possible reinsertions (Figure~\ref{fig:reinsertion-cases}). The new vertex may be attached to either endpoint of \(e\), and for each endpoint it may be inserted on one of the two boundary sides adjacent to that endpoint. We eliminate three of these four cases.

\textbf{Step 1:} The arrow direction eliminates two reinsertions.  Since \(\psi\) sends \(s_1\) to \(s_2\), these two arrows have the same direction. Hence \(s_1\) is last at \(q_1\) if and only if \(s_2\) is last at \(q_2\), and inserting \(p_j\) into the other boundary side at a given endpoint would reverse the direction of \(\alpha_j\). This eliminates two of the four possibilities, leaving one candidate at each endpoint of \(e\).

\textbf{Step 2:} The relations at \(e\) eliminate one of the two remaining reinsertions. Assume \(s_j\) is last at \(q_j\), the case in which \(s_j\) is first at \(q_j\) being dual. Write \(e = \{u,v\}\), where by Step 1, one candidate attaches  at \(u\) and the other at \(v\), such that \(q_j \in \{u,v\}\) for \(j = 1,2\).

Suppose first that at one of these endpoints, say \(u\), some edge \(a'\) of \(S\) follows \(e\) in the counterclockwise order there. Attaching the new edge at \(u\) would place \(a'\) between the new edge and \(e\), so there would be no arrow between them (Panel (e) of Figure~\ref{fig:reinsertion-cases}). This contradicts \(e\) being adjacent to \(s_j\), so that candidate is eliminated, \(q_1 = q_2 = v\), and \(T_1 = T_2\) by Step 1. So we may assume \(e\) is the last edge of \(S\) in the counterclockwise order at both \(u\) and \(v\), so that both candidates produce an arrow between the new edge and \(e\).

Now suppose \(q_1 \neq q_2\), so that \(\{q_1,q_2\} = \{u,v\}\). One of these endpoints, say \(q_2\), must be incident to another edge of \(S\). Let \(a\) be the edge of \(S\) immediately preceding \(e\) in the counterclockwise order at \(q_2\), which exists because
\(e\) is last there. Then \(F(S)\) has an arrow \(a \to e\) arising at \(q_2\). The two arrows of the path \(a \to e \to s_1\) arise at \(q_2\) and at \(q_1\), which are distinct, so by Definition~\ref{def: F} this path is a relation in \(F(T_1)\). But both arrows of \(a \to e \to s_2\) arise at \(q_2\), so that path is not a relation in \(F(T_2)\). Further, \(\psi\) fixes \(a\) and \(e\) and sends \(\alpha_1\) to \(\alpha_2\),
so it sends the first path to the second, contradicting that \(\psi\) preserves relations. Hence \(q_1 = q_2\), and by Step 1 the two reinsertions insert the new vertex into the same boundary side there, so \(T_1 = T_2\).

\end{proof}

\begin{figure}[h]
\centering
\begin{tikzpicture}[scale=0.8,
  polyvert/.style={circle, fill=black, inner sep=1.2pt},
  polyedge/.style={gray!55, thin},
  treeedge/.style={thick, black},
  eedge/.style={thick, ForestGreen},
  aedge/.style={thick, orange!90!black},
  sedge/.style={thick, red},
  lbl/.style={font=\small},
  elbl/.style={font=\scriptsize, black},
  panelcap/.style={font=\bfseries\small},
  sub/.style={font=\footnotesize, align=center, text width=3.5cm}]

\begin{scope}[shift={(7,0)}]
  \node[polyvert] (p) at (45 + 90:1.1)  {};  \node[lbl] at (45 + 90:1.42)  {$u$};
  \node[polyvert] (q) at (45 - 0:1.1)   {};  \node[lbl] at (45 + 0:1.42)   {$v$};
  \node[polyvert] (r) at (45 - 90:1.1) {};  \node[lbl] at (45 - 90:1.42) {$r$};
  \node[polyvert] (t) at (45 +180:1.1) {};  \node[lbl] at (45 + 180:1.42) {$t$};

  \foreach \i/\j in {p/q,q/r,r/t,t/p} {\draw[polyedge] (\i) -- (\j);}

  \draw[eedge] (p) to[bend right=30] (q) node[pos=0.5, xshift=0pt, yshift=10pt, elbl] {$e$};
  \draw[treeedge] (q) -- (t);
  \draw[treeedge] (r) to[bend right=30] (t);

  \node[panelcap] at (0,-1.75) {(a) $S=T_j/s_j$};
\end{scope}

\begin{scope}[shift={(0,-4.5)}]
  \node[polyvert] (p) at (72+90:1.1)   {};  \node[lbl] at (72+90:1.42)   {$u$};
  \node[polyvert] (x) at (72+18:1.1)   {};  \node[lbl] at (72+18:1.5)    {$x$};
  \node[polyvert] (q) at (72-54:1.1)  {};  \node[lbl] at (72-54:1.48)  {$v$};
  \node[polyvert] (r) at (72-126:1.1) {};  \node[lbl] at (72-126:1.48) {$r$};
  \node[polyvert] (t) at (72+162:1.1)  {};  \node[lbl] at (72+162:1.48)  {$t$};

  \foreach \i/\j in {p/x,x/q,q/r,r/t,t/p} {\draw[polyedge] (\i) -- (\j);}

  \draw[sedge] (x) to[bend left=30] (p) node[pos=0.5, xshift=0pt, yshift=20pt, elbl] {$s$};
  \draw[eedge] (p) to[bend right=0] (q) node[pos=0.7, xshift=0pt, yshift=3pt, elbl] {$e$};
  \draw[treeedge] (q) -- (t);
  \draw[treeedge] (r) to[bend right=30] (t);

  \node[panelcap] at (0,-1.75) {(b)};
  \node[sub] at (0,-2.5) {Surviving reinsertion};
\end{scope}

\begin{scope}[shift={(5,-4.5)}]
  \node[polyvert] (p) at (72+90:1.1)   {};  \node[lbl] at (72+90:1.42)   {$u$};
  \node[polyvert] (x) at (72+18:1.1)   {};  \node[lbl] at (72+18:1.5)    {$x$};
  \node[polyvert] (q) at (72-54:1.1)  {};  \node[lbl] at (72-54:1.48)  {$v$};
  \node[polyvert] (r) at (72-126:1.1) {};  \node[lbl] at (72-126:1.48) {$r$};
  \node[polyvert] (t) at (72+162:1.1)  {};  \node[lbl] at (72+162:1.48)  {$t$};

  \foreach \i/\j in {p/x,x/q,q/r,r/t,t/p} {\draw[polyedge] (\i) -- (\j);}

  \draw[sedge] (x) to[bend right=30] (q) node[pos=0.5, xshift=0pt, yshift=20pt, elbl] {$s$};
  \draw[eedge] (p) to[bend right=0] (q) node[pos=0.7, xshift=0pt, yshift=3pt, elbl] {$e$};
  \draw[treeedge] (q) -- (t);
  \draw[treeedge] (r) to[bend right=30] (t);

  \node[panelcap] at (0,-1.75) {(c)};
  \node[sub] at (0,-2.5) {Step 1 - wrong arrow direction};
\end{scope}

\begin{scope}[shift={(10,-4.5)}]
  \node[polyvert] (p) at (0+90:1.1)   {};  \node[lbl] at (0+90:1.42)   {$u$};
  \node[polyvert] (q) at (0+18:1.1)   {};  \node[lbl] at (0+18:1.42)   {$v$};
  \node[polyvert] (r) at (0-54:1.1)  {};  \node[lbl] at (0-54:1.42)  {$r$};
  \node[polyvert] (t) at (0-126:1.1) {};  \node[lbl] at (0-126:1.42) {$t$};
  \node[polyvert] (x) at (0+162:1.1)  {};  \node[lbl] at (0+162:1.5)   {$x$};

  \foreach \i/\j in {p/q,q/r,r/t,t/x,x/p} {\draw[polyedge] (\i) -- (\j);}

  \draw[sedge] (x) to[bend right=30] (p) node[pos=0.5, xshift=-10pt, yshift=9pt, elbl] {$s$};
  \draw[eedge] (p) to[bend right=30] (q) node[pos=0.5, xshift=10pt, yshift=9pt, elbl] {$e$};
  \draw[treeedge] (q) -- (t);
  \draw[treeedge] (r) to[bend right=30] (t);

  \node[panelcap] at (0,-1.75) {(d)};
  \node[sub] at (0,-2.5) {Step 1 - wrong arrow direction};
\end{scope}

\begin{scope}[shift={(15,-4.5)}]
  \node[polyvert] (p) at (72+90:1.1)   {};  \node[lbl] at (72+90:1.42)   {$u$};
  \node[polyvert] (q) at (72+18:1.1)   {};  \node[lbl] at (72+18:1.42)   {$v$};
  \node[polyvert] (x) at (72-54:1.1)  {};  \node[lbl] at (72-54:1.5)   {$x$};
  \node[polyvert] (r) at (72-126:1.1) {};  \node[lbl] at (72-126:1.42) {$r$};
  \node[polyvert] (t) at (72+162:1.1)  {};  \node[lbl] at (72+162:1.42)  {$t$};

  \foreach \i/\j in {p/q,q/x,x/r,r/t,t/p} {\draw[polyedge] (\i) -- (\j);}

  \draw[sedge] (q) to[bend right=30] (x) node[pos=0.5, xshift=8pt, yshift=12pt, elbl] {$s$};
  \draw[eedge] (p) to[bend right=30] (q) node[pos=0.5, xshift=-17pt, yshift=6pt, elbl] {$e$};
  \draw[treeedge] (q) -- (t);
  \draw[treeedge] (r) to[bend right=30] (t);

  \node[panelcap] at (0,-1.75) {(e)};
  \node[sub] at (0,-2.5) {Step 2 - no arrow between s and e};
\end{scope}

\end{tikzpicture}
\caption{Example reinsertion of tree from Figure~\ref{fig:deletion-reinsertion}. Red arcs indicate the reinserted boundary leaf edge $s$ and green arcs its unique neighbor $e$. Panels (b)-(e) show the four ways to reinsert $s$ so that the new vertex $x$ is attached to an endpoint of $e$. Two are eliminated in Step~1 because they reverse the direction of the arrow at $s$, and one is eliminated in Step~2 because an edge of $S$ then separates $s$ from $e$, leaving no arrow between them.}
\label{fig:reinsertion-cases}
\end{figure}
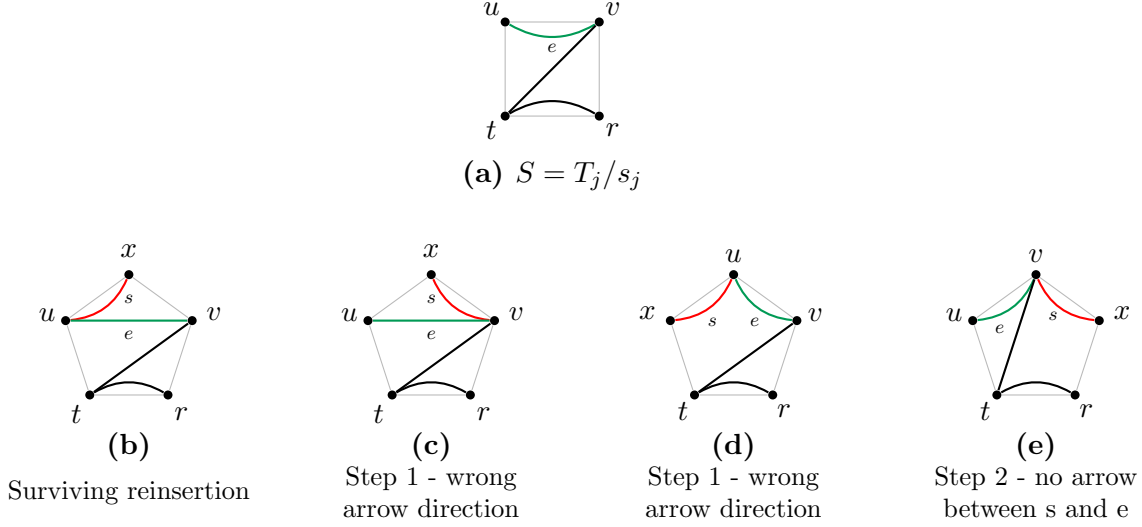

\subsection{Inductive Reconstruction}

\begin{thm}\label{thm:comb}
Let \(T,T'\) be non-crossing spanning trees in the \((n+1)\)-gon. Then
\[
F(T)\cong F(T')
\]
as quivers with relations if and only if
\[
T'=\sigma^m(T)
\]
for some cyclic rotation \(\sigma^m\). Further, every isomorphism \(\varphi: F(T) \to F(T')\) is induced by such a rotation.
\end{thm}

\begin{proof}
The forward direction is immediate. A cyclic rotation preserves counterclockwise order, consecutive incidence at each polygon vertex, and the cyclic order condition defining the relations, so a rotation \(\sigma^m\) with \(T' = \sigma^m(T)\) induces an isomorphism \(F(T) \cong F(T')\).

For the converse, we argue by induction on \(n\), proving the stronger statement that every isomorphism \(\varphi: F(T) \to F(T')\) is induced by a cyclic rotation \(\sigma^m\) with \(T' = \sigma^m(T)\).

For the triangle, every non-crossing spanning tree has two edges, and all such trees are cyclic rotations of one another; the only isomorphisms between their \(F(T)\)'s are those induced by the rotations identifying them.

Let \(n \ge 3\), suppose the result holds for \(n\)-gons, and let \(\varphi:F(T)\to F(T')\) be an isomorphism, where \(T, T'\) are non-crossing spanning trees in the \((n+1)\)-gon.

By Lemma~\ref{lem:boundary-leaf-detect}, there exists a boundary leaf edge \(s \subset T\) such that \(s' := \varphi(s)\) is a boundary leaf edge of \(T'\). By Lemma~\ref{lem:induced-isomorphism-after-deletion}, \(\varphi\) induces an isomorphism
\[
\overline{\varphi}: F(T/s) \to F(T'/s')
\]
of quivers with relations on the \(n\)-gon.

By the inductive hypothesis applied to \(\overline{\varphi}\), there is a cyclic rotation \(\sigma^m\) of the \(n\)-gon with \(\sigma^m(T/s) = T'/s'\), and \(\overline{\varphi}\) is induced by \(\sigma^m\).

We lift \(\sigma^m\) to the \((n+1)\)-gon. Let \(x\) be the endpoint of \(s\) incident to no other edge of \(T\). Deleting \(x\) identifies the polygon vertices of the \((n+1)\)-gon other than \(x\), in cyclic order, with the polygon vertices of the \(n\)-gon, and under this identification \(x\) lies between two adjacent polygon vertices \(i\) and \(i+1\) of the \(n\)-gon. Let \(\tilde{\sigma}^m\) be the bijection on the polygon vertices of the \((n+1)\)-gon which agrees with \(\sigma^m\) under this identification and which sends \(x\) between \(\sigma^m(i)\) and \(\sigma^m(i+1)\). Then \(\tilde{\sigma}^m\) preserves the cyclic order, so it is a cyclic rotation of the \((n+1)\)-gon.

Set \(T_1 := \tilde{\sigma}^m(T)\) and \(s_1 := \tilde{\sigma}^m(s)\), a boundary leaf edge of \(T_1\). By construction \(T_1/s_1 = \sigma^m(T/s) = T'/s'\), and the isomorphism \(\Phi: F(T) \to F(T_1)\) induced by \(\tilde{\sigma}^m\) restricts on \(F(T/s)\) to the map induced by \(\sigma^m\). Hence
\[
\psi := \varphi \circ \Phi^{-1} : F(T_1) \to F(T')
\]
is an isomorphism of quivers with relations with \(\psi(s_1) = s'\) which induces the identity on \(F(T_1/s_1) = F(T'/s')\).

By Lemma~\ref{lem:reinsertion-unique}, \(T_1 = T'\). Therefore \(T' = \tilde{\sigma}^m(T)\) and \(\varphi = \Phi\), so \(\varphi\) is induced by the rotation \(\tilde{\sigma}^m\).
\end{proof}

\subsection{Proof of Theorem \ref{thm:main}}

\begin{proof}[Proof of Theorem \ref{thm:main}]
By Proposition~\ref{prop:F-HomExt},
\[
(Q^\chi,R^\chi) \cong F(T(\chi))
\]
and
\[
(Q^{\chi'},R^{\chi'}) \cong F(T(\chi')).
\]
Therefore
\[
(Q^\chi,R^\chi) \cong (Q^{\chi'},R^{\chi'})
\]
if and only if
\[
F(T(\chi))\cong F(T(\chi')).
\]
By Theorem~\ref{thm:comb}, this holds if and only if
\[
T(\chi')=\sigma^m(T(\chi)).
\]
\end{proof}

Combining Theorems \ref{thm: bij bw H-E and ITA} and \ref{thm:main}, we obtain the following. 

\begin{thm}\label{thm: bij bw ITA and trees up to rotation}
    Iterated tilted algebras in type $A_n$ up to isomorphism are in bijection with non-crossing spanning trees on a convex $n+1$-gon up to rotation, which are counted by the following formula (OEIS A296532) 

    \[a(n)=
    \begin{cases}
        \frac{\binom{3n}{n}}{(n+1)(2n+1)} & \text{if } n \text{ is even}\\ 
         & \\
        \frac{1}{n+1}\bigg(\frac{\binom{3n}{n}}{2n+1} + \left( {{3n-1\over 2} \atop {n-1\over 2}} \right) \bigg)& \text{if } n \text{ is odd.}
    \end{cases}\]
   \qed
\end{thm}

\subsection{Formula for Exceptional Sequences in Type \(A_n\)}\label{subsec: counting E.S. in type A}

Combining Theorem~\ref{thm:main} with the linear extension formula, Theorem 3.4 in \cite{igusa2025hom}, we obtain an explicit formula for the number of complete exceptional sequences in mod-$\Lambda$, where $\Lambda$ is the path algebra of a quiver of type $A_n$.

For a complete exceptional collection $\chi$ in mod-$\Lambda$, let $\mathrm{LinExt}(Q^\chi, R^\chi)$ denote the number of linear extensions of the partial order $\leq_e$ on $(Q^\chi, R^\chi)$ defined in \cite[Definition 3.4]{igusa2025hom}, and let $\mathrm{Aut}(Q^\chi, R^\chi)$ denote the group of automorphisms of $(Q^\chi, R^\chi)$ as a quiver with relations. Let $e(\Lambda)$ denote the number of complete exceptional sequences in mod-$\Lambda$, which, by the results of \cite{obaid2013number}, is $e(\Lambda)=(n+1)^{n-1}$.

\begin{thm}\label{thm:expseq-count}
The number of complete exceptional sequences in mod-$\Lambda$ is
\[
e(\Lambda) = h(\Lambda) \sum_{[(Q,R)]} \frac{\mathrm{LinExt}(Q,R)}{|\mathrm{Aut}(Q,R)|},
\]
where the sum is over isomorphism classes of Hom--Ext quivers $(Q^\chi, R^\chi)$ of complete exceptional collections $\chi$ in mod-$\Lambda$, and $h(\Lambda) = n+1$ is the Coxeter number of $\Lambda$.
\end{thm}

\begin{proof}
By \cite[Theorem 3.4]{igusa2025hom}, the complete exceptional sequences with underlying collection \(\chi\) are in bijection with the linear extensions of a well-defined partial order, \(\leq_e\), on \((Q^\chi, R^\chi)\). So
\[
e(\Lambda) = \sum_{\chi} \mathrm{LinExt}(Q^\chi, R^\chi),
\]
where the sum is over complete exceptional collections \(\chi\) in mod-$\Lambda$. By Araya's bijection \cite{araya2013exceptional}, this equals
\[
\sum_{T} \mathrm{LinExt}(Q^T, R^T),
\]
where the sum is over non-crossing spanning trees \(T\) of the \((n+1)\)-gon and \((Q^T, R^T) := F(T)\).

The cyclic rotation group \(\langle \sigma \rangle \cong \mathbb{Z}/(n+1)\mathbb{Z}\) acts on this set of trees and by Theorem~\ref{thm:comb}, two trees \(T, T'\) satisfy \(F(T) \cong F(T')\) if and only if they lie in the same orbit (differ by a cyclic rotation \(\sigma^m\)). Thus, the orbits are in bijection with isomorphism classes of \((Q^\chi, R^\chi)\). For a tree \(T\), the orbit-stabilizer theorem gives orbit size \((n+1)/|\mathrm{Stab}_{\langle\sigma\rangle}(T)|\). By the second clause of Theorem~\ref{thm:comb}, every automorphism of \(F(T)\) is induced by a unique rotation fixing \(T\). Hence \(\mathrm{Aut}(F(T)) \cong \mathrm{Stab}_{\langle\sigma\rangle}(T)\).

Grouping the sum by orbits, each orbit contributes
\[
\frac{n+1}{|\mathrm{Aut}(Q^T, R^T)|} \cdot \mathrm{LinExt}(Q^T, R^T),
\]
and summing over isomorphism classes gives the stated formula.
\end{proof}

\bibliographystyle{amsalpha}
\bibliography{bibliography}
\end{document}